\documentclass[12pt,reqno]{amsart}

\usepackage{amssymb,amsmath,graphicx,amsfonts,euscript}
\usepackage{color}

\newtheorem{thm}{Theorem}[section]

\newtheorem{prop}[thm]{Proposition}
\newtheorem{define}[thm]{Definition}

\numberwithin{equation}{section}
\begin{document}
\title[Hall-MHD Equations]{Local well-posedness for the Hall-MHD equations with fractional magnetic diffusion}

\author[D. Chae, R. Wan and J. Wu]{Dongho Chae$^{1}$, Renhui Wan$^{2}$ and  Jiahong Wu$^{3}$}

\address{$^1$ Department of Mathematics,
College of Natural Science,
Chung-Ang University,
Seoul 156-756, Republic of Korea}

\email{dchae@cau.ac.kr}

\address{$^2$ Department of Mathematics,
Zhejiang University,
Hanzhou 310027, China}

\email{rhwanmath@163.com,21235002@zju.edu.cn}

\address{$^3$ Department of Mathematics,
Oklahoma State University,
Stillwater, OK 74078, USA; and
Department of Mathematics,
Chung-Ang University,
Seoul 156-756, Republic of Korea}

\email{jiahong@math.okstate.edu}

\vskip .2in
\begin{abstract}
The Hall-magnetohydrodynamics (Hall-MHD) equations, rigorously derived from kinetic models,
are useful in describing many physical phenomena in geophysics and astrophysics. This paper
studies the local well-posedness of classical solutions to the Hall-MHD equations with
the magnetic diffusion given by a fractional Laplacian operator, $(-\Delta)^\alpha$. Due to
the presence of the Hall term in the Hall-MHD equations, standard energy estimates appear to
indicate that we need $\alpha\ge 1$ in order to obtain the local well-posedness. This paper breaks the barrier and shows that the fractional
Hall-MHD equations are locally well-posed for any $\alpha>\frac12$. The approach here fully exploits
the smoothing effects of the dissipation and establishes the local bounds for the Sobolev norms
through the Besov space techniques. The method presented here
may be applicable to similar situations involving other partial differential equations.
\end{abstract}

\maketitle

\vskip .2in
\section{Introduction}

This paper focuses on the Hall-magnetohydrodynamics (Hall-MHD) equations with fractional magnetic diffusion,
\begin{equation} \label{HallMHD}
\left\{
\begin{array}{l}
\partial_t u + u\cdot\nabla u + \nabla p = B\cdot\nabla B,  \\
\partial_t B + u\cdot\nabla B + \nabla\times((\nabla\times B)\times B) + (-\Delta)^\alpha B =B\cdot\nabla u,\\
\nabla \cdot u=0,\quad \nabla \cdot B=0, \\
u(x,0) =u_0(x), \quad B(x,0) =B_0(x),
\end{array}
\right.
\end{equation}
where $x\in \mathbb{R}^d$ with $d\ge 2$,  $u=u(x,t)$ and $B=B(x,t)$ are vector fields representing
the velocity and the magnetic field, respectively, $p=p(x,t)$ denotes the pressure, $\alpha>0$
is a parameter and the fractional Laplacian $(-\Delta)^\alpha$ is defined through the Fourier transform,
$$
\widehat{(-\Delta)^\alpha f} (\xi) = |\xi|^{2\alpha} \, \widehat{f} (\xi).
$$
For notational convenience, we also use $\Lambda$ for $(-\Delta)^{\frac12}$. The Hall-MHD equations
with the usual Lapalcian dissipation were derived in \cite{Degond4} from kinetic models.
The Hall-MHD equations differ from the standard incompressible MHD equations in the Hall
term $\nabla\times((\nabla\times B)\times B)$, which is important in the study of magnetic
reconnection (see, e.g., \cite{HG,War}). The Hall-MHD equations have been mathematically
investigated in several works (\cite{Degond4,CDL,CL,CS,CW}).
Global weak solutions of (\ref{HallMHD}) with both $\Delta u$ and $\Delta B$ and local classical
solutions of (\ref{HallMHD}) with $\Delta B$ (with or without $\Delta u$) were
obtained in \cite{CDL}. In addition, a blowup criterion and the global existence
of small classical solutions were also established in
\cite{CDL}. These results were later sharpened by \cite{CL}.

\vskip .1in
We examine the issue of whether or not (\ref{HallMHD}) is locally well-posedness when the fractional
power $\alpha<1$.  Previously local solutions of (\ref{HallMHD}) were obtained
for $\alpha=1$ (\cite{CDL,CL}). Standard energy estimates appear to indicate that $\alpha \ge 1$
is necessary in order to obtain local bounds for the solutions in Sobolev spaces. This requirement comes from the
estimates of the regularity-demanding Hall term $\nabla\times((\nabla\times B)\times B)$. To understand more
precisely the issue at hand, we perform a short energy estimate on the essential part of the
equation for $B$,
$$
\partial_t B + \nabla\times((\nabla\times B)\times B) + (-\Delta)^\alpha B =0.
$$
The global $L^2$-bound
\begin{equation} \label{l2in}
\|B(t)\|_{L^2}^2 + 2\, \int_0^t \|\Lambda^\alpha B(\tau)\|_{L^2}^2 d\tau = \|B_0\|_{L^2}^2
\end{equation}
follows from the simple fact
\begin{align*}
\int \nabla\times((\nabla\times B)\times B) \cdot B
=\int ((\nabla\times B)\times B)\cdot (\nabla\times B)=0.
\end{align*}
To obtain the $H^1$-bound, we invoke the equation for $\|\nabla B\|_{L^2}^2$,
\begin{align*}
\frac12 \frac{d}{dt} \|\nabla B\|_{L^2}^2  + \|\Lambda^\alpha \nabla B\|_{L^2}^2 =
-\sum_{i=1}^d \int \partial_i \nabla\times((\nabla\times B)\times B) \cdot \partial_i B.
\end{align*}
We can indeed shift one-derivative, namely
\begin{align*}
\int \partial_i \nabla\times((\nabla\times B)\times B) \cdot \partial_i B
= \int ((\nabla\times B) \times \partial_i B) \cdot \partial_i \nabla\times B.
\end{align*}
H\"{o}lder's inequality allows us to conclude that
\begin{align*}
\frac12 \frac{d}{dt} \|\nabla B\|_{L^2}^2  + \|\Lambda^\alpha \nabla B\|_{L^2}^2
\le \|\nabla B\|_{L^2}\,\|\nabla B\|_{L^\infty}\,\|\nabla\nabla\times B\|_{L^2}.
\end{align*}
Therefore, it appears that we need $\alpha\ge 1$ in order to bound the term
$\|\nabla\nabla\times B\|_{L^2}$ on the
right-hand side. More generally, the energy inequality involving the $H^\sigma$-norm
\begin{align*}
\frac{d}{dt} \|B\|_{H^\sigma}^2  + \|\Lambda^{\alpha} B\|_{H^\sigma}^2
\le C\, \| B\|_{H^\sigma}\,\|\nabla B\|_{L^\infty}\,\|\nabla B\|_{H^\sigma}.
\end{align*}
also appears to demand that $\alpha \ge 1$ in order to bound $\|\nabla B\|_{H^\sigma}$.

\vskip .1in
This paper obtains the local existence and
uniqueness of solutions to (\ref{HallMHD}) with any $\alpha>\frac12$.
More precisely, we prove the following theorem.

\begin{thm} \label{main}
Consider (\ref{HallMHD}) with $\alpha>\frac12$. Assume $(u_0, B_0)\in H^\sigma(\mathbb{R}^d)$
with $\sigma>1+ \frac{d}2$, and $\nabla\cdot u_{0}=\nabla\cdot B_{0}=0$. Then there exist $T_0=T_0(\|(u_0, B_0)\|_{H^\sigma})>0$ and a unique
solution $(u,B)$ of (\ref{HallMHD}) on $[0,T_0]$ such that
$$
(u, B) \in L^\infty([0,T_0]; H^\sigma(\mathbb{R}^d)).
$$
In addition, for any $\sigma'<\sigma$,
$$
(u, B) \in C([0,T_0]; H^{\sigma'}(\mathbb{R}^d))
$$
and $\|(u(t), B(t))\|_{H^\sigma}$ is continuous from the right on $[0,T_0)$.
\end{thm}

The essential idea of proving Theorem \ref{main} is to fully
exploit the dissipation in the equation for $B$ and
estimate the Sobolev norm $\|(u, B)\|_{H^\sigma}$ via
Besov space techniques. We identify $H^\sigma$ with the Besov
space $B^\sigma_{2,2}$ and suitably shift the derivatives
in the nonlinear term. The definition of Besov spaces and related facts
used in this paper are provided in the appendix.
The rest of this paper is divided into two sections followed by an appendix.
Section \ref{sec:localbd} states and proves the result for the local {\it a priori} bound. Section \ref{sec:mainproof} presents the
complete proof of Theorem \ref{main}. The appendix supplies the definitions of the Littlewood-Paley decomposition and Besov spaces.

\vskip .3in
\section{Local a priori bound}
\label{sec:localbd}

This section establishes a local {\it a priori} bound for smooth solutions of (\ref{HallMHD}), which is
the key component in the proof of Theorem \ref{main}. The result for the local {\it a priori} bound can be stated
as follows.
\begin{prop} \label{localbd}
Consider (\ref{HallMHD}) with $\alpha>\frac12$. Assume the initial data $(u_0, B_0)
\in H^\sigma(\mathbb{R}^d)$ with $\sigma >1+ \frac{d}{2}$. Let $(u,B)$ be the corresponding
solution. Then, there exists $T_0=T_0(\|(u_0, B_0)\|_{H^\sigma})>0$ such that, for $t\in [0,T_0]$,
$$
\|(u(t), B(t))\|_{H^\sigma} \le C(\alpha, T_0, \|(u_0, B_0)\|_{H^\sigma})
$$
and
$$
\int_0^{T_0} \|\Lambda^\alpha B(s)\|_{H^\sigma}^2\,ds \le C(\alpha, T_0, \|(u_0, B_0)\|_{H^\sigma}).
$$
\end{prop}

\begin{proof}[Proof of Proposition \ref{localbd}] The proof
identifies the Sobolev space $H^\sigma$ with the
Besov space $B^\sigma_{2,2}$ and resorts to Besov space techniques.

\vskip .1in
Let $l\ge -1$ be an integer and let $\Delta_l$ denote the homogeneous frequency localized operator.   Applying $\Delta_l$ to (\ref{HallMHD}) yields
\begin{align*}
& \partial_t \Delta_l u + \Delta_l(u\cdot\nabla u) + \nabla \Delta_l p = \Delta_l(B\cdot\nabla B), \\
& \partial_t \Delta_l B + \Delta_l(u\cdot\nabla B) + \Delta_l \nabla\times((\nabla\times B)\times B)
+ (-\Delta)^\alpha \Delta_l B =\Delta_l(B\cdot\nabla u).
\end{align*}
Taking the inner product with $(\Delta_l u, \Delta_l B)$ and integrating by parts, we have
\begin{align}
\frac12 \frac{d}{dt} \left(\|\Delta_l u\|_{L^2}^2 + \|\Delta_l B\|_{L^2}^2\right) + C_0 2^{2\alpha l} \|\Delta_l B\|_{L^2}^2
= K_1 + K_2 + K_3 + K_4 + K_5, \label{rootnode}
\end{align}
where
\begin{align*}
&K_1 = -\int [\Delta_l, u\cdot\nabla] u\cdot \Delta_l u, \qquad K_2 = -\int [\Delta_l, u\cdot\nabla] B\cdot \Delta_l B,\\
&K_3 = \int [\Delta_l, B\cdot\nabla] B\cdot \Delta_l u, \qquad K_4 = \int [\Delta_l, B\cdot\nabla] u\cdot \Delta_l B,\\
&K_5 =  - \int \Delta_l \nabla\times((\nabla\times B)\times B) \cdot \Delta_l B.
\end{align*}
Note that we have used the standard commutator notation,
$$
[\Delta_l, u\cdot\nabla] u = \Delta_l(u\cdot\nabla u) - u\cdot\nabla (\Delta_lu)
$$
and applied the lower bound, for a constant $C_0>0$,
$$
\int \Delta_l B \cdot (-\Delta)^\alpha \Delta_l B \ge C_0 2^{2\alpha l} \|\Delta_l B\|_{L^2}^2.
$$
Using the notion of paraproducts, we write
$$
K_1=K_{11} + K_{12} + K_{13},
$$
where
\begin{align*}
K_{11} =& \sum_{|k-l|\le 2} \int \left(\Delta_l(S_{k-1} u \cdot \nabla\Delta_k u) - S_{k-1} u \cdot\nabla\Delta_l\Delta_k u\right)\,\cdot \Delta_lu, \\
K_{12} =& \sum_{|k-l|\le 2} \int \left(\Delta_l(\Delta_k u \cdot \nabla S_{k-1} u) - \Delta_k u \cdot\nabla\Delta_lS_{k-1} u
\right)\,\cdot \Delta_lu, \\
K_{13}= & \sum_{k\ge l-1} \int \left(\Delta_l(\Delta_k u \cdot \nabla\widetilde{\Delta}_k u) -
\Delta_k u \cdot \nabla\Delta_l\widetilde{\Delta}_k u\right)\,\cdot \Delta_lu
\end{align*}
with $\widetilde{\Delta}_k =\Delta_{k-1} + \Delta_k + \Delta_{k+1}$.
By H\"{o}lder's inequality and a standard commutator estimate,
\begin{align*}
|K_{11}| \le & C\,\|\nabla S_{l-1} u\|_{L^\infty} \, \|\Delta_l u\|_{L^2} 
\sum_{|k-l|\le 2} \|\Delta_k u\|_{L^2}\\
 \le &  C\, \|\nabla u\|_{L^\infty} \, \|\Delta_l u\|_{L^2}
\sum_{|k-l|\le 2} \|\Delta_k u\|_{L^2}.
\end{align*}
Since the summation over $k$ for fixed $l$ above consists of only a finite 
number of terms and, as we shall later in the proof, the norm generated by each term is a multiple of that generated by the typical term,  it suffices to keep the typical term with $k=l$ and ignore the summation. This would 
help keep our presentation concise. We will invoke this practice throughout
the rest of the paper. 
By H\"{o}lder's inequality, $K_{12}$ is bounded by
$$
|K_{12}| \le C\, \|\nabla u\|_{L^\infty}\,\|\Delta_l u\|^2_{L^2}.
$$
By H\"{o}lder's inequality and Bernstein's inequality,
$$
|K_{13}| \le C\, \|\Delta_l u\|_{L^2}\,\|\nabla u\|_{L^\infty} \, \sum_{k\ge l-1} 2^{l-k}\, \|\Delta_k u\|_{L^2}.
$$
Therefore,
\begin{align*}
|K_1|  \le C\, \|\Delta_l u\|_{L^2}\, \|\nabla u\|_{L^\infty} \,\left(\|\Delta_l u\|_{L^2} + \sum_{k\ge l-1} 2^{l-k}\, \|\Delta_k u\|_{L^2}\right).
\end{align*}
Similarly, $K_2$, $K_3$ and $K_4$ are bounded by
\begin{align*}
|K_2|  \le& \,C\, \|\nabla u\|_{L^\infty} \,\|\Delta_l B\|^2_{L^2} +  C\, \|\nabla B\|_{L^\infty} \,\|\Delta_l u\|_{L^2}\|\Delta_l B\|_{L^2}\\
&+\, C\,\|\nabla u\|_{L^\infty} \,\|\Delta_l B\|_{L^2}\sum_{k\ge l-1} 2^{l-k}\, \|\Delta_k B\|_{L^2},
\end{align*}
\begin{align*}
|K_3|  \le C\, \|\nabla B\|_{L^\infty} \,\|\Delta_l u\|_{L^2}\,\,\left(\|\Delta_l B\|_{L^2}
+ \sum_{k\ge l-1} 2^{l-k}\, \|\Delta_k B\|_{L^2}\right),
\end{align*}
\begin{align*}
|K_4|  \le &\, C\, \|\nabla B\|_{L^\infty}\, \|\Delta_l u\|_{L^2}\,\|\Delta_l B\|_{L^2}\,+ C\, \|\nabla u\|_{L^\infty} \,\|\Delta_l B\|^2_{L^2} \\
&+ C\, \|\nabla B\|_{L^\infty}\,\|\Delta_l u\|_{L^2} \sum_{k\ge l-1} 2^{l-k}\, \|\Delta_k B\|_{L^2}.
\end{align*}
Using the simple fact
$$
(B\times (\Delta_l\nabla\times B))\cdot \Delta_l\nabla\times B =0
$$
and the vector identity 
$$
B\times (\nabla\times B)= \frac12 \nabla (B\cdot B) -(B\cdot\nabla) B,
$$
we can rewrite $K_5$ as 
\begin{align}
 K_5 =& \int  \left(\Delta_l(B\times (\nabla\times B))-B\times (\Delta_l\nabla\times B)\right) \cdot \Delta_l\nabla\times B \notag\\
 =& -\int  [\Delta_l, B\cdot\nabla] B 
    \cdot \Delta_l\nabla\times B \label{K3like}\\
  & + \int  \left(\Delta_l\left(\frac12 \nabla (B\cdot B)\right)- (\nabla \Delta_l B)\cdot B\right) \cdot \Delta_l\nabla\times B. \label{K5like}
\end{align}
The term in (\ref{K3like}) can be estimated in a similar way 
as $K_3$. To estimate the term in (\ref{K5like}), we use the the 
notion of paraproducts to write
$$
\int  \left(\Delta_l\left(\frac12 \nabla (B\cdot B)\right)- (\nabla \Delta_l B)\cdot B\right) \cdot \Delta_l\nabla\times B = K_{51} + K_{52} + K_{53},
$$
where,
\begin{align*}
K_{51} =& \sum_{|k-l|\le 2} \int \left(\Delta_l((\nabla S_{k-1} B)\cdot \Delta_k B) - (\nabla \Delta_l S_{k-1} B) \cdot\Delta_k B\right)\,\cdot \Delta_l\nabla\times B, \\
K_{52} =& \sum_{|k-l|\le 2} \int \left(\Delta_l(S_{k-1} B \cdot (\nabla \Delta_k B)) - (S_{k-1}B) \cdot (\nabla \Delta_l\Delta_k B)
\right)\,\cdot \Delta_l\nabla\times B, \\
K_{53} = & \sum_{k\ge l-1} \int \left(\Delta_l\left(\nabla\left(\frac12\Delta_k B \cdot \widetilde{\Delta}_k B\right)\right)
- (\nabla\Delta_l\Delta_k B) \cdot \widetilde{\Delta}_k B\right)\,\cdot \Delta_l\nabla\times B.
\end{align*}
By H\"{o}lder's inequality,
\begin{align*}
|K_{51}| \le & \|\Delta_l((\nabla S_{k-1} B)\cdot \Delta_k B) - (\nabla \Delta_l S_{k-1} B) \cdot\Delta_k B\|_{L^2} \,
\|\Delta_l\nabla\times B\|_{L^2} \\
\le & C\,2^l\,\|\nabla S_{l-1} B\|_{L^\infty} \, \|\Delta_l B\|^2_{L^2} \le C\,2^l\,\|\nabla B\|_{L^\infty} \, \|\Delta_l B\|^2_{L^2}.
\end{align*}
By H\"{o}lder's inequality and a standard commutator estimate,
$$
|K_{52}| \le C\,2^l\,\|\nabla B\|_{L^\infty} \, \|\Delta_l B\|^2_{L^2}.
$$
By H\"{o}lder's inequality and Bernstein's inequality,
$$
|K_{53}| \le C\, 2^l\,\|\nabla B\|_{L^\infty} \, \|\Delta_l B\|_{L^2}  \sum_{k\ge l-1} 2^{l-k}\, \|\Delta_k B\|_{L^2}.
$$
Therefore,
\begin{align*}
|K_5| \le C\, 2^l\,\|\nabla B\|_{L^\infty} \, \|\Delta_l B\|_{L^2} \left(\|\Delta_l B\|_{L^2}
+ \sum_{k\ge l-1} 2^{l-k}\, \|\Delta_k B\|_{L^2}\right).
\end{align*}
Inserting the estimates above in (\ref{rootnode}), we obtain
\begin{align*}
& \frac{d}{dt} \left(\|\Delta_l u\|_{L^2}^2 + \|\Delta_l B\|_{L^2}^2\right) + C_0 2^{2\alpha l} \|\Delta_l B\|_{L^2}^2 \\
& \quad \le C\,\|(\nabla u, \nabla B)\|_{L^\infty} \,(\|\Delta_l u\|_{L^2}^2 +\|\Delta_l B\|^2_{L^2})\\
&\quad\,\,+ C\,\|(\nabla u, \nabla B)\|_{L^\infty}\,\left[\left(\sum_{k\ge l-1} 2^{l-k}\, \|\Delta_k u\|_{L^2}\right)^2 + \left(\sum_{k\ge l-1} 2^{l-k}\, \|\Delta_k B\|_{L^2}\right)^2\right]\\
& \quad  + \, C\, 2^{l}
\|\nabla B\|_{L^\infty} \, \|\Delta_l B\|^2_{L^2}+ C\, 2^{l} \|\nabla B\|_{L^\infty}\, \|\Delta_l B\|_{L^2} \, \sum_{k\ge l-1} 2^{l-k}\, \|\Delta_k B\|_{L^2}.
\end{align*}
Multiplying the inequality above by
$2^{2\sigma l}$ and summing over $l\ge -1$, invoking the global bound for the $L^2$-norm of $(u, B)$ and the equivalence of the norms
$$
\|f\|_{H^\sigma}^2 \quad \sim \sum_{l\ge -1} 2^{2\sigma l}\, \|\Delta_l f\|_{L^2}^2,
$$
we have
\begin{eqnarray}
&& \|u(t)\|^2_{H^\sigma} + \|B(t)\|^2_{H^\sigma}
+ \,C_0\, \int_0^t \|B(\tau)\|^2_{H^{\sigma+\alpha}}\,d\tau \notag \\
&& \quad \le \|u_0\|^2_{H^\sigma} + \|B_0\|^2_{H^\sigma}
+ C\,\int_0^t \|(\nabla u, \nabla B)\|_{L^\infty}\,(\|u(\tau)\|^2_{H^\sigma} + \|B(\tau)\|^2_{H^\sigma})\,d\tau \notag\\
&& \qquad + \,C\, \sum_{l\ge -1} 2^{(2\sigma+1) l}\,\int_0^t \|\nabla B\|_{L^\infty} \,\|\Delta_l B\|^2_{L^2}\,d\tau  \notag\\
&& \qquad + \,C\, \sum_{l\ge -1} 2^{(2\sigma+1) l}\,\int_0^t \|\nabla B\|_{L^\infty} \,\left(\sum_{k\ge l-1} 2^{l-k}\, \|\Delta_k B\|_{L^2}\right)^2\,d\tau.
\label{hsin}
\end{eqnarray}
To derive the inequality above, we have used Young's inequality for series convolution
\begin{align*}
\sum_{l\ge -1} 2^{2\sigma l} \left(\sum_{k\ge l-1} 2^{l-k}\, \|\Delta_k u\|_{L^2}\right)^2
=&\sum_{l\ge -1} \left(\sum_{k\ge l-1} 2^{(\sigma +1)(l-k)}\, 2^{\sigma k}\|\Delta_k u\|_{L^2}\right)^2\\
\le & C\, \sum_{l\ge -1} 2^{2\sigma l}\|\Delta_l u\|^2_{L^2} \le  C\, \|u\|_{H^\sigma}^2.
\end{align*}
We further bound the last two terms in \eqref{hsin},
\begin{align*}
L_1 \equiv&  \,C\, \sum_{l\ge -1} 2^{(2\sigma+1) l}\,\int_0^t \|\nabla B\|_{L^\infty} \,\|\Delta_l B\|^2_{L^2}\,d\tau,\\
L_2 \equiv&  \,C\, \sum_{l\ge -1} 2^{(2\sigma+1) l}\,\int_0^t \|\nabla B\|_{L^\infty} \,\left(\sum_{k\ge l-1} 2^{l-k}\, \|\Delta_k B\|_{L^2}\right)^2\,d\tau.
\end{align*}
Set $\theta=1-\frac1{2\alpha}$. For $\alpha>\frac12$, $\theta \in (0,1)$.
By H\"{o}lder's inequality,
\begin{align*}
L_1 = & \,C\, \int_0^t \|\nabla B\|_{L^\infty} \sum_{l\ge -1}\,\left(2^{2\sigma l}\|\Delta_l B\|^{2}_{L^2}\right)^\theta\,\,\,\left(2^{2(\sigma +\alpha)l}\, \|\Delta_l B\|^{2}_{L^2}\right)^{(1-\theta)}\,d\tau\\
 \le & \,C\, \int_0^t \|\nabla B\|_{L^\infty} \left(\sum_{l\ge -1}2^{2\sigma l}\|\Delta_l B\|^{2}_{L^2}\right)^\theta\,\,\left(\sum_{l\ge -1}2^{2(\sigma +\alpha)l}\, \|\Delta_l B\|^{2}_{L^2}\right)^{(1-\theta)}\,d\tau \\
\le & \,C\, \int_0^t \|\nabla B\|^{\frac1\theta}_{L^\infty} \|B\|_{H^\sigma}^2\,d\tau + \frac{C_0}{4} \int_0^t\|B(\tau)\|^2_{H^{\sigma+\alpha}} \,d\tau.
\end{align*}
By Young's inequality for series convolution and an interpolation inequality,
\begin{align*}
L_2 = & \,C\, \int_0^t \|\nabla B\|_{L^\infty} \sum_{l\ge -1}
\left(\sum_{k\ge l-1} 2^{(l-k)(\sigma +\frac12)}\, 2^{(\sigma+\frac12)k}\|\Delta_k B\|_{L^2}\right)^2 \,d\tau\\
\le & \,C\, \int_0^t \|\nabla B\|_{L^\infty}\,\|B\|_{H^{\sigma+\frac12}}^2\,d\tau
\le \,C\, \int_0^t \|\nabla B\|_{L^\infty}\,
\|B\|_{H^{\sigma}}^{2\theta}\,\|B\|_{H^{\sigma+\alpha}}^{2(1-\theta)}\,d\tau \\
\le & \,C\, \int_0^t \|\nabla B\|^{\frac1\theta}_{L^\infty} \|B\|_{H^\sigma}^2\,d\tau + \frac{C_0}{4} \int_0^t\|B(\tau)\|^2_{H^{\sigma+\alpha}} \,d\tau.
\end{align*}
Inserting the estimates above in \eqref{hsin} and invoking the embedding inequalities
\begin{equation*}
\|\nabla B\|_{L^\infty} \le C\, \|B\|_{H^\sigma} \qquad \mbox{for $\sigma >1 + \frac{d}{2}$},
\end{equation*}
we have
\begin{eqnarray}
&& \|u(t)\|^2_{H^\sigma} + \|B(t)\|^2_{H^\sigma}
+ \,C_0\, \int_0^t \|B(\tau)\|^2_{H^{\sigma+\alpha}}\,d\tau \notag \\
&& \quad \le \|u_0\|^2_{H^\sigma} + \|B_0\|^2_{H^\sigma}
+ C\,\int_0^t \left(\|u(t)\|^2_{H^\sigma} + \|B(t)\|^2_{H^\sigma}\right)^\gamma\,d\tau,
\end{eqnarray}
for a constant $\gamma>1$. This inequality implies a local bound for $\|u(t)\|^2_{H^\sigma} + \|B(t)\|^2_{H^\sigma}$, namely for some $T_0=T_0(\|(u_0, B_0)\|_{H^\sigma})>0$ such that, for $t\in [0,T_0]$,
$$
\|u(t)\|_{H^\sigma} + \|B(t)\|_{H^\sigma} \le C(u_0, B_0,\alpha, T_0)
$$
and
\begin{equation} \label{timebd}
\int_0^{T_0}\|B(\tau)\|_{H^{\sigma+\alpha}}^2 \,d\tau <\infty.
\end{equation}
This completes the proof of Proposition \ref{localbd}.
\end{proof}

\vskip .3in
\section{Local existence and uniqueness}
\label{sec:mainproof}

This section proves Theorem \ref{main}.

\begin{proof}[Proof of Theorem \ref{main}]
The local existence and uniqueness can be obtained through an approximation procedure. Here we use the Friedrichs method,  a smoothing approach through filtering the high frequencies. For each positive integer $n$, we define
$$
\widehat{\mathcal{J}_n f}(\xi) = \chi_{B_n}(\xi) \, \widehat{f}(\xi),
$$
where $B_n$ denotes the closed ball of radius $n$ centered at $0$ and $\chi_{B_n}$ denotes the
characteristic functions on $B_n$. Denote
$$
H_n^\sigma \equiv\left\{f \in H^\sigma(\mathbb{R}^d), \,\, \mbox{supp}\widehat{f} \subset B_n \right\}.
$$
We seek a solution $(u, B) \in H_n^\sigma$ satisfying
\begin{equation}\label{appmhd}
\left\{
\begin{array}{l}
\partial_t u + \mathcal{J}_n\mathcal{P}(\mathcal{J}_n\mathcal{P}u\cdot\nabla \mathcal{J}_n\mathcal{P} u)
= \mathcal{J}_n\mathcal{P}(\mathcal{J}_n\mathcal{P}B\cdot\nabla \mathcal{J}_n\mathcal{P} B),  \\
\partial_t B + \mathcal{J}_n\mathcal{P}(\mathcal{J}_n\mathcal{P}u\cdot\nabla \mathcal{J}_n\mathcal{P} B) + \mathcal{J}_n\mathcal{P}(\nabla\times((\nabla\times \mathcal{J}_n\mathcal{P} B)\times \mathcal{J}_n\mathcal{P}B))  \\
 \qquad \qquad + (-\Delta)^\alpha \mathcal{J}_n\mathcal{P} B =\mathcal{J}_n\mathcal{P}(\mathcal{J}_n\mathcal{P}B\cdot\nabla \mathcal{J}_n\mathcal{P} u),\\
 u(x,0) =(\mathcal{J}_n u_0)(x), \quad B(x,0) =(\mathcal{J}_n B_0)(x),
\end{array}
\right.
\end{equation}
where $\mathcal{P}$ denotes the projection onto divergence-free vector fields.

\vskip .1in
For each fixed $n\ge 1$, it is not very hard, although tedious, to verify that the right-hand side of
(\ref{appmhd}) satisfies the Lipschitz condition in $H_n^\sigma$ and, by Picard's theorem, (\ref{appmhd})
has a unique global (in time) solution. The uniqueness implies that
$$
\mathcal{J}_n\mathcal{P}u =u, \quad \mathcal{J}_n\mathcal{P} B =B
$$
and ensures the divergence-free conditions $\nabla\cdot u=0$ and $\nabla\cdot B=0$. Then,
(\ref{appmhd}) is simplified to
$$
\left\{
\begin{array}{l}
\partial_t u + \mathcal{J}_n\mathcal{P}(u\cdot\nabla  u)
= \mathcal{J}_n\mathcal{P}(B\cdot\nabla  B),  \\
\partial_t B + \mathcal{J}_n\mathcal{P}(u\cdot\nabla  B)
+ \mathcal{J}_n\mathcal{P}(\nabla\times((\nabla\times B)\times B))
+ (-\Delta)^\alpha  B =\mathcal{J}_n\mathcal{P}(B\cdot\nabla u).
\end{array}
\right.
$$
We denote this solution by $(u^n, B^n)$. As in the proof of Proposition \ref{localbd}, we can show that $(u^n, B^n)$ satisfies
\begin{align}
\|(u^n, B^n)\|^2_{H^\sigma} \le \|(u_0^n, B_0^n)\|^2_{H^\sigma}
+ C\, \int_0^t \|(u^n(s), B^n(s))\|^{2\gamma}_{H^\sigma}\,ds\label{hsigma}
\end{align}
for some $\gamma>1$. Due to $\|(u_0^n, B_0^n)\|_{H^\sigma} \le \|(u_0, B_0)\|_{H^\sigma}$, this
inequality is uniform in $n$. This allows us to obtain a uniform local bound
\begin{align}
\sup_{t\in[0,T_0]} \|(u^n(t), B^n(t))\|_{H^\sigma} \le M(\alpha, T_0, \|(u_0, B_0)\|_{H^\sigma}).
\label{unifbd}
\end{align}
As in (\ref{timebd}), we also have the uniform local bound for the time integral
$$
\int_0^{T_0} \|(\Lambda^\alpha B^n)(s)\|^2_{H^\sigma}\,ds \le M(\alpha, T_0, \|(u_0, B_0)\|_{H^\sigma}).
$$
Furthermore, these uniform bounds allow us to show that
\begin{align}
\|(u^n, B^n) -(u^m, B^m)\|_{L^2} \to 0 \qquad\mbox{as $n,m\to \infty$}.  \label{strongl2}
\end{align}
This is shown through standard energy estimates for $\|(u^n, B^n) -(u^m, B^m)\|_{L^2}$.
The process involves many terms, but most of them can be handled in a standard
fashion (see, e.g., \cite[p.107]{MB}). We provide the detailed energy estimate for the term that is special here, namely the Hall term $\nabla\times((\nabla\times B)\times B)$. In the process of the energy estimates, we
need to bound the term
\begin{align}
& \int \left(\nabla\times((\nabla\times B^n)\times B^n)-\nabla\times((\nabla\times B^m)\times B^m)\right)
\cdot (B^n-B^m)\,dx   \notag\\
&\quad = \int \left(\nabla\times((\nabla\times (B^n-B^m))\times B^n)\right)
\cdot (B^n-B^m)\,dx    \notag\\
&\quad\quad + \int \left(\nabla\times((\nabla\times B^m)\times (B^n-B^m)\right)
\cdot (B^n-B^m)\,dx. \label{energy}
\end{align}
The first term on the right-hand side of (\ref{energy}) is zero,
\begin{align*}
& \int \left(\nabla\times((\nabla\times (B^n-B^m))\times B^n)\right)
\cdot (B^n-B^m)\,dx \\
&\quad = \int ((\nabla\times (B^n-B^m))\times B^n)\cdot (\nabla\times(B^n-B^m))\,dx =0.
\end{align*}
For the second term on the right of (\ref{energy}), by the simple vector identity
\begin{align*}
&\nabla\times((\nabla\times B^m)\times (B^n-B^m)) \\
&\qquad\qquad = (B^n-B^m)\cdot\nabla(\nabla\times B^m)- (\nabla\times B^m)\cdot\nabla(B^n-B^m),
\end{align*}
we have
\begin{align*}
& \left| \int \left(\nabla\times((\nabla\times B^m)\times (B^n-B^m)\right)
\cdot (B^n-B^m)\,dx \right| \\
&\qquad \quad \le \|\nabla(\nabla\times B^m)\|_{L^\frac{d}{\alpha}}\|B^n-B^m\|_{L^{\frac{2d}{d-2\alpha}}}\|B^n-B^m\|_{L^2}\\
&\qquad \quad \le C\|\Lambda^\alpha B^m\|_{H^\sigma}^2\|B^n-B^m\|_{L^2}^2+\frac{1}{8}\|\Lambda^\alpha (B^n-B^m)\|_{L^2}^2,
\end{align*}
where we have used
$$\|\nabla^2 f\|_{L^\frac{d}{\alpha}}\le C\|\Lambda^\alpha f\|_{H^\sigma},\ \ \|f\|_{L^\frac{2d}{d-2\alpha}}\le C\|\Lambda^\alpha f\|_{L^2}.$$
Putting together the estimates for all the terms, we obtain
\begin{align*}
\frac{d}{dt} \|B^n-B^m\|_{L^2}^2 \le  C \|\Lambda^\alpha B^m\|_{H^\sigma}^2\|B^n-B^m\|_{L^2}^2 + C\,\left(\frac1n + \frac1m\right).
\end{align*}
Noticing that $\|\Lambda^\alpha B^m\|_{H^\sigma}^2$ is time integrable, Gronwall's inequality yields
the desired convergence (\ref{strongl2}). Let $(u,B)$ be the limit. Due to the
uniform bound (\ref{unifbd}), $(u,B)\in H^\sigma$ for $t\in [0,T_0]$. By the interpolation
inequality, for any $0<\sigma' <\sigma$,
$$
\|f\|_{H^{\sigma'}} \le C_{\sigma}\,\|f\|^{1-\frac{\sigma'}{\sigma}}_{L^{2}}\, \|f\|^{\frac{\sigma'}{\sigma}}_{H^{\sigma}},
$$
we further obtain the strong convergence
$$
\|(u^n, B^n) -(u, B)\|_{H^{\sigma'}} \to 0\qquad\mbox{as $n\to \infty$}
$$
and consequently, $(u, B)\in C([0,T_0]; H^{\sigma'})$. This strong convergence makes it easy to check
that $(u, B)$ satisfies the Hall-MHD equation in (\ref{HallMHD}). In addition, the time continuity in $(u, B)\in C([0,T_0]; H^{\sigma'})$ allows to show the weak time continuity
$$
(u, B)\in C_W([0,T_0]; H^\sigma)\quad\mbox{or}\quad t\,\mapsto \int (u(x,t), B(x,t)) \cdot \phi(x)\,dx \quad \mbox{is continuous}
$$
for any $\phi\in H^{-\sigma}$.
To show the right (in time) continuity of $\|(u(t), B(t))\|_{H^\sigma}$, we make use
of the energy inequality, for any $t>\widetilde{t}$,
\begin{align*}
\|(u(t), B(t))\|^2_{H^\sigma} \le& \|(u(\widetilde{t}), B(\widetilde{t}))\|^2_{H^\sigma}
+ C\, \int_{\widetilde{t}}^t \|(u(s), B(s))\|^\gamma_{H^{2\sigma}}\,ds,
\end{align*}
This inequality can be obtained in a similar fashion as (\ref{hsigma}). Then,
$$
\lim_{t\to \widetilde{t}+} \|(u(t), B(t))\|_{H^\sigma}
\le \|(u(\widetilde{t}), B(\widetilde{t}))\|_{H^\sigma}.
$$
By the weak continuity in time,
$$
\|(u(\widetilde{t}), B(\widetilde{t}))\|_{H^\sigma} \le \lim_{t\to \widetilde{t}+} \|(u(t), B(t))\|_{H^\sigma}.
$$
The desired right (in time) continuity of $\|(u(t), B(t))\|_{H^\sigma}$ then follows. This completes the
proof of Theorem \ref{main}.
\end{proof}

\vskip .3in
\appendix
\section{Besov spaces}
\label{Besov}

This appendix provides the definitions of some of the functional spaces and related facts
used in the previous sections. Materials presented in this appendix can be found in several
books and many papers (see, e.g., \cite{BCD,BL,MWZ,RS,Tri}).

\vskip .1in
We start with several notations. $\mathcal{S}$ denotes
the usual Schwartz class and ${\mathcal S}'$ its dual, the space of
tempered distributions. ${\mathcal S}_0$ denotes a subspace of ${\mathcal
S}$ defined by
$$
{\mathcal S}_0 = \left\{ \phi\in {\mathcal S}: \,\, \int_{\mathbb{R}^d}
\phi(x)\, x^\gamma \,dx =0, \,|\gamma| =0,1,2,\cdots \right\}
$$
and ${\mathcal S}_0'$ denotes its dual. ${\mathcal S}_0'$ can be identified
as
$$
{\mathcal S}_0' = {\mathcal S}' / {\mathcal S}_0^\perp = {\mathcal S}' /{\mathcal P}
$$
where ${\mathcal P}$ denotes the space of multinomials.

\vspace{.1in} To introduce the Littlewood-Paley decomposition, we
write for each $j\in \mathbb{Z}$
\begin{equation*}\label{aj}
A_j =\left\{ \xi \in \mathbb{R}^d: \,\, 2^{j-1} \le |\xi| <
2^{j+1}\right\}.
\end{equation*}
The Littlewood-Paley decomposition asserts the existence of a
sequence of functions $\{\Phi_j\}_{j\in {\mathbb Z}}\subset{\mathcal S}$ such
that
$$
\mbox{supp} \widehat{\Phi}_j \subset A_j, \qquad
\widehat{\Phi}_j(\xi) = \widehat{\Phi}_0(2^{-j} \xi)
\quad\mbox{or}\quad \Phi_j (x) =2^{jd} \Phi_0(2^j x),
$$
and
$$
\sum_{j=-\infty}^\infty \widehat{\Phi}_j(\xi) = \left\{
\begin{array}{ll}
1&,\quad \mbox{if}\,\,\xi\in {\mathbb R}^d\setminus \{0\},\\
0&,\quad \mbox{if}\,\,\xi=0.
\end{array}
\right.
$$
Therefore, for a general function $\psi\in {\mathcal S}$, we have
$$
\sum_{j=-\infty}^\infty \widehat{\Phi}_j(\xi)
\widehat{\psi}(\xi)=\widehat{\psi}(\xi) \quad\mbox{for $\xi\in {\mathbb
R}^d\setminus \{0\}$}.
$$
In addition, if $\psi\in {\mathcal S}_0$, then
$$
\sum_{j=-\infty}^\infty \widehat{\Phi}_j(\xi)
\widehat{\psi}(\xi)=\widehat{\psi}(\xi) \quad\mbox{for any $\xi\in
{\mathbb R}^d $}.
$$
That is, for $\psi\in {\mathcal S}_0$,
$$
\sum_{j=-\infty}^\infty \Phi_j \ast \psi = \psi
$$
and hence
$$
\sum_{j=-\infty}^\infty \Phi_j \ast f = f, \qquad f\in {\mathcal S}_0'
$$
in the sense of weak-$\ast$ topology of ${\mathcal S}_0'$. For
notational convenience, we define
\begin{equation}\label{del1}
\mathring{\Delta}_j f = \Phi_j \ast f, \qquad j \in {\mathbb Z}.
\end{equation}

\begin{define}
For $s\in {\mathbb R}$ and $1\le p,q\le \infty$, the homogeneous Besov
space $\mathring{B}^s_{p,q}$ consists of $f\in {\mathcal S}_0' $
satisfying
$$
\|f\|_{\mathring{B}^s_{p,q}} \equiv \|2^{js} \|\mathring{\Delta}_j
f\|_{L^p}\|_{l^q} <\infty.
$$
\end{define}

\vspace{.1in}
We now choose $\Psi\in {\mathcal S}$ such that
$$
\widehat{\Psi} (\xi) = 1 - \sum_{j=0}^\infty \widehat{\Phi}_j (\xi),
\quad \xi \in {\mathbb R}^d.
$$
Then, for any $\psi\in {\mathcal S}$,
$$
\Psi \ast \psi + \sum_{j=0}^\infty \Phi_j \ast \psi =\psi
$$
and hence
\begin{equation*}\label{sf}
\Psi \ast f + \sum_{j=0}^\infty \Phi_j \ast f =f
\end{equation*}
in ${\mathcal S}'$ for any $f\in {\mathcal S}'$. To define the inhomogeneous Besov space, we set
\begin{equation} \label{del2}
\Delta_j f = \left\{
\begin{array}{ll}
0,&\quad \mbox{if}\,\,j\le -2, \\
\Psi\ast f,&\quad \mbox{if}\,\,j=-1, \\
\Phi_j \ast f, &\quad \mbox{if} \,\,j=0,1,2,\cdots.
\end{array}
\right.
\end{equation}
\begin{define}
The inhomogeneous Besov space $B^s_{p,q}$ with $1\le p,q \le \infty$
and $s\in {\mathbb R}$ consists of functions $f\in {\mathcal S}'$
satisfying
$$
\|f\|_{B^s_{p,q}} \equiv \|2^{js} \|\Delta_j f\|_{L^p} \|_{l^q}
<\infty.
$$
\end{define}

\vskip .1in
The Besov spaces $\mathring{B}^s_{p,q}$ and $B^s_{p,q}$ with  $s\in (0,1)$ and $1\le p,q\le \infty$ can be equivalently defined by the norms
$$
\|f\|_{\mathring{B}^s_{p,q}}  = \left(\int_{\mathbb{R}^d} \frac{(\|f(x+t)-f(x)\|_{L^p})^q}{|t|^{d+sq}} dt\right)^{1/q},
$$
$$
\|f\|_{B^s_{p,q}}  = \|f\|_{L^p} + \left(\int_{\mathbb{R}^d} \frac{(\|f(x+t)-f(x)\|_{L^p})^q}{|t|^{d+sq}} dt\right)^{1/q}.
$$
When $q=\infty$, the expressions are interpreted in the normal way.

\vskip .1in
Many frequently used function spaces are special cases of Besov spaces. The following proposition
lists some useful equivalence and embedding relations.
\begin{prop}
For any $s\in \mathbb{R}$,
$$
\mathring{H}^s \sim \mathring{B}^s_{2,2}, \quad H^s \sim B^s_{2,2}.
$$
For any $s\in \mathbb{R}$ and $1<q<\infty$,
$$
\mathring{B}^{s}_{q,\min\{q,2\}} \hookrightarrow \mathring{W}_{q}^s \hookrightarrow \mathring{B}^{s}_{q,\max\{q,2\}}.
$$
In particular, $\mathring{B}^{0}_{q,\min\{q,2\}} \hookrightarrow L^q \hookrightarrow \mathring{B}^{0}_{q,\max\{q,2\}}$.
\end{prop}

\vskip .1in
For notational convenience, we write $\Delta_j$ for
$\mathring{\Delta}_j$. There will be no confusion if we keep in mind that
$\Delta_j$'s associated with the homogeneous Besov spaces is defined in
(\ref{del1}) while those associated with the inhomogeneous Besov
spaces are defined in (\ref{del2}). Besides the Fourier localization operators $\Delta_j$,
the partial sum $S_j$ is also a useful notation. For an integer $j$,
$$
S_j \equiv \sum_{k=-1}^{j-1} \Delta_k,
$$
where $\Delta_k$ is given by (\ref{del2}). For any $f\in \mathcal{S}'$, the Fourier
transform of $S_j f$ is supported on the ball of radius $2^j$.

\vskip .1in
Bernstein's inequalities are useful tools in dealing with Fourier localized functions
and these inequalities trade integrability for derivatives. The following proposition
provides Bernstein type inequalities for fractional derivatives.
\begin{prop}\label{bern}
Let $\alpha\ge0$. Let $1\le p\le q\le \infty$.
\begin{enumerate}
\item[1)] If $f$ satisfies
$$
\mbox{supp}\, \widehat{f} \subset \{\xi\in \mathbb{R}^d: \,\, |\xi|
\le K 2^j \},
$$
for some integer $j$ and a constant $K>0$, then
$$
\|(-\Delta)^\alpha f\|_{L^q(\mathbb{R}^d)} \le C_1\, 2^{2\alpha j +
j d(\frac{1}{p}-\frac{1}{q})} \|f\|_{L^p(\mathbb{R}^d)}.
$$
\item[2)] If $f$ satisfies
\begin{equation*}\label{spp}
\mbox{supp}\, \widehat{f} \subset \{\xi\in \mathbb{R}^d: \,\, K_12^j
\le |\xi| \le K_2 2^j \}
\end{equation*}
for some integer $j$ and constants $0<K_1\le K_2$, then
$$
C_1\, 2^{2\alpha j} \|f\|_{L^q(\mathbb{R}^d)} \le \|(-\Delta)^\alpha
f\|_{L^q(\mathbb{R}^d)} \le C_2\, 2^{2\alpha j +
j d(\frac{1}{p}-\frac{1}{q})} \|f\|_{L^p(\mathbb{R}^d)},
$$
where $C_1$ and $C_2$ are constants depending on $\alpha,p$ and $q$
only.
\end{enumerate}
\end{prop}

\vskip .4in
\section*{Acknowledgements}
Chae was partially supported by NRF grant No.2006-0093854 and No.2009-0083521. Wu was partially supported by NSF grant DMS1209153 and the AT\&T Foundation at Oklahoma State University.


\vskip .4in
\begin{thebibliography}{99}

\bibitem{Degond4} M. Acheritogaray, P. Degond, A. Frouvelle and J.-G. Liu, Kinetic formulation and global existence for the Hall-Magneto-hydrodynamics system, {\it Kinet. Relat. Models  \bf 4} (2011), 901-918.

\bibitem{BCD} H. Bahouri, J.-Y. Chemin and R. Danchin, \textit{Fourier Analysis and Nonlinear Partial Differential Equations}, Springer, 2011.

\bibitem{BL} J. Bergh and J. L\"{o}fstr\"{o}m, {\it Interpolation Spaces,
An Introduction}, Springer-Verlag, Berlin-Heidelberg-New York, 1976.


\bibitem{CDL}D. Chae, P. Degond and J.-G. Liu, Well-posedness for Hall-magnetohydrodynamics,  {\it Ann. Inst. H. Poincar\'{e} Anal. Non Lin\'{e}aire \bf 31} (2014),  555-565.


\bibitem{CL} D. Chae and J. Lee, On the blow-up criterion and small data global existence for the Hall-magnetohydrodynamics,  {\it J.  Differential Equations \bf 256} (2014), 3835-3858.

\bibitem{CS} D. Chae and M. Schonbek,  On the temporal decay for the Hall-magnetohydrodynamic equations, {\it J. Differential Equations \bf 255} (2013), 3971-3982.

\bibitem{CW} D. Chae and S. Weng, Singularity formation for the incompressible Hall-MHD equations without resistivity,  arXiv:1312.5519 [math.AP] 19 Dec 2013.

\bibitem{HG}H. Homann and R. Grauer, Bifurcation analysis of magnetic reconnection in Hall-MHD systems,
{\it Phys. D \bf 208} (2005), 59-72.

\bibitem{MB} A.J. Majda and A.L. Bertozzi, \textit{Vorticity and Incompressible Flow}, Cambridge University Press, Cambridge, UK, 2001.

\bibitem{MWZ}C. Miao, J. Wu and Z. Zhang, \textit{Littlewood-Paley Theory and its Applications
in Partial Differential Equations of Fluid Dynamics}, Science Press, Beijing, China, 2012 (in Chinese).

\bibitem{RS} T. Runst and W. Sickel, \textit{Sobolev Spaces of fractional order,
Nemytskij operators and Nonlinear Partial Differential Equations}, Walter de Gruyter, Berlin, New York, 1996.

\bibitem{Tri} H. Triebel, \textit{Theory of Function Spaces II}, Birkhauser Verlag, 1992.


\bibitem{War} M. Wardle, Star formation and Hall effect, {\it Astrophysics and Space Science \bf 292} (2004), 317-323.

\end{thebibliography}
\end{document}